\documentclass{amsart}
\usepackage{amsmath,amssymb,amsthm,mathrsfs,mathtools}
\usepackage[colorlinks,citecolor=red,pagebackref,hypertexnames=false,breaklinks]{hyperref}
\usepackage{tikz}

\theoremstyle{definition}
\newtheorem{dfn}{Definition}[section]
\newtheorem{thm}[dfn]{Theorem}
\newtheorem{prop}[dfn]{Proposition}
\newtheorem{lem}[dfn]{Lemma}
\newtheorem{cor}[dfn]{Corollary}
\newtheorem{remark}[dfn]{Remark}

\numberwithin{equation}{section}

\title[Single logarithmic stability for quantitative unique continuation]{New quantitative unique continuation result for elliptic equations}

\author{Mourad Choulli}
\address{Universit\'{e} de Lorraine, 34 cours L\'{e}opold, 54052 Nancy cedex, France}
\email{mourad.choulli@univ-lorraine.fr}

\author{Hiroshi Takase}
\address{Department of Mathematics, Okayama University, 3-1-1 Tsushima-naka, Kita-ku, Okayama 700-8530, Japan}
\email{takase@math.okayama-u.ac.jp}

\date{\today}
\begin{document}
\begin{abstract}
We prove a new quantitative unique continuation result for elliptic equations from Cauchy data. We provide a simple and direct proof based only on a Carleman inequality. Similar result for the Stokes equation is also shown.
\end{abstract}

\maketitle

\section{Introduction}\label{section1}

The (strong) unique continuation for elliptic equations (from Cauchy or interior data) is an old problem and it is completely solved even for equations with unbounded lower-order coefficients (e.g., Koch-Tataru \cite{Koch2001a}). However, the optimal quantitative version of the unique continuation remains an open problem in the case of a $C^{0,1}$ domain and an $H^2$ solution (or less). By optimal, we mean a stability inequality with a single logarithmic continuity modulus. An optimal stability inequality for a $C^{0,1}$ domain and a $C^{1,\alpha}$ solution has been established by Bourgeois-Dard\'{e} \cite{Bourgeois2010a} (Laplace operator) and Choulli \cite{Choulli2016} (general case) with an  improved proof in Bellassoued-Choulli \cite{Bellassoued2023}. The case of a $C^{1,1}$ domain and an $H^2$ solution has been obtained in Bourgeois \cite{Bourgeois2010} (Laplace operator) and Choulli \cite{Choulli2024a} (general case extended to an $H^\eta$ solution, $3/2<\eta\le 2$, and semilinear elliptic equations). Other results for second-order partial differential equations, including observability inequalities, can be found in Choulli \cite{Choulli}. A non-optimal stability inequality has been established by Choulli \cite{Choulli2020} for  a $C^{0,1}$ domain and an $H^2$ solution. Specifically, the stability inequality in \cite{Choulli2020} has a double logarithmic continuity modulus.

In this short note, we relax in the case of an $H^2$ solution the $C^{1,1}$ regularity  of the domain by imposing a particular geometric configuration. But the main novelty is that we provide a direct and simple proof. We remind the reader that the proofs given in the reference cited above are quite technical and are broken down into three steps. More precisely, it is a matter of quantifying the continuation from the boundary to an interior subdomain, the continuation from an interior subdomain to another interior subdomain and the continuation from an interior subdomain to the boundary.

Before stating our main result, we introduce the definitions and the notations we need. Let $n\ge 2$ be an integer, $B\subset\mathbb{R}^n$ be a $C^2$ bounded open set with boundary $\mathcal{S}$ and $\Omega\Supset B$ be a $C^{0,1}$ bounded domain such that $D:=\Omega\setminus\overline{B}$ is connected. The boundary of $\Omega$ will be denoted by $\Gamma$. 

Throughout this text, we use the Einstein summation convention for quantities with indices. If in any term appears twice, as both an upper and lower index, that term is assumed to be summed from $1$ to $n$. Henceforth, all functions are assumed to be complex-valued unless otherwise specified.

Let $(g_{k\ell})\in W^{1,\infty}(D;\mathbb{R}^{n\times n})$ be a symmetric matrix-valued function satisfying, for some  $\rho>0$, 
\[
g_{k\ell}(x)\xi^k\xi^\ell\ge \rho|\xi|^2\quad x\in D,\; \xi=(\xi^1,\ldots,\xi^n) \in\mathbb{R}^n.
\]
Note that $(g^{k\ell})$ the matrix inverse to $(g_{k\ell})$ is uniformly positive definite as well. Recall that the Laplace-Beltrami operator associated with  the metric tensor $g=g_{k\ell}dx^k\otimes dx^\ell$ is given by
\[
\Delta_g u:=(1/\sqrt{|g|})\partial_k\left(\sqrt{|g|}g^{k\ell}\partial_\ell u\right),
\]
where $|g|=\mbox{det} (g)$. For convenience, we recall the following usual notations
\begin{align*}
&\langle X,Y\rangle=g_{k\ell}X^kY^\ell,\quad X=X^k\frac{\partial}{\partial x_k},\quad Y=Y^k\frac{\partial}{\partial  x_k},
\\
&\nabla_g w=g^{k\ell}\partial_k w\frac{\partial}{\partial x_\ell},
\\
&|\nabla_g w|_g^2=\langle\nabla_g w,\nabla_g \bar{w}\rangle=g^{k\ell}\partial_k w\partial_\ell \bar{w}, 
\\
&\nu_g=(\nu_g)^k\frac{\partial}{\partial x_k},\quad (\nu_g)^k=\frac{g^{k\ell}\nu_\ell}{\sqrt{g^{\alpha\beta}\nu_\alpha\nu_\beta}},
\\
&\partial_{\nu_g}w=\langle\nu_g,\nabla_g w\rangle, 
\end{align*}
where $\nu$ denotes the outer unit normal vector field to $\partial D$. Also, define the tangential gradient $\nabla_{\tau_g} w$ with respect to $g$ by
\[
\nabla_{\tau_g} w:=\nabla_g w-(\partial_{\nu_g} w)\nu_g.
\]
We find that $|\nabla_{\tau_g} w|_g^2=|\nabla_g w|_g^2-|\partial_{\nu_g} w|^2$ holds.

We define the operator $P$ by
\[
Pu:=-\Delta_g u+\langle X,\nabla_g u\rangle+p u,
\]
where $(X^1,\ldots,X^n)\in L^\infty(D;\mathbb{C}^n)$ and $p\in L^r(D)$, where $r:=n$ when $n\ge 3$ and $r>2$ when $n=2$.

It is worth noting that the magnetic Laplace-Beltrami operator can be written in the form $P$. We recall that the magnetic Laplace-Beltrami operator is an operator of the form
\[
L=(1/\sqrt{|g|})(\partial_k+ia_k)\sqrt{|g|}g^{k\ell}(\partial_\ell+ia_\ell),
\]
where $a=(a_1,\ldots ,a_n)\in W^{1,\infty}(D,\mathbb{R}^n)$.

The volume form in $D$ with respect to $g$ will be denoted by $dV_g:=\sqrt{|g|}dx$ and $dS_g:=\sqrt{|g|}dS$ denotes the surface element on $\partial D$ with respect to $g$.

Let
\[
\mathcal{C}(u):=\|Pu\|_{L^2(D)}+\|u\|_{H^1(\Gamma)}+\|\partial_{\nu_g}u\|_{L^2(\Gamma)},
\]
where
\begin{align*}
&\|Pu\|_{L^2(D)}:=\left(\int_D|Pu|^2dV_g\right)^{1/2},
\\
&\|u\|_{H^1(\Gamma)}:=\left(\int_\Gamma(|\nabla_{\tau_g}u|_g^2+|u|^2)dS_g\right)^{1/2},
\\
&\|\partial_{\nu_g}u\|_{L^2(\Gamma)}:=\left(\int_\Gamma|\partial_{\nu_g}u|^2dS_g\right)^{1/2}.
\end{align*}

The main result of this note is the following theorem.

\begin{thm}\label{interpolation_inequality}
Let $0\le\eta<2$ and $\zeta:=(g,X,p,B,\Omega,\eta)$. Then there exist constants $\mathbf{c}=\mathbf{c}(\zeta)>0$ and $c=c(\zeta) >0$ such that for any $u\in H^2(D)$ and $s\ge 1$ we have
\begin{equation}\label{quc}
\mathbf{c}\|u\|_{H^\eta(D)}\le e^{cs}\mathcal{C}(u)+s^{-(2-\eta)/2}\|u\|_{H^2(D)}.
\end{equation}
\end{thm}

It is worth remarking that Theorem \ref{interpolation_inequality} quantifies the unique continuation from Cauchy data on $\Gamma$. Indeed, if $u\in H^2(\Omega)$ satisfies $\mathcal{C}(u)=0$, then from \eqref{quc} $\|u\|_{H^\eta(D)}\le s^{-(2-\eta)/2}\|u\|_{H^2(D)}$ for all $s\ge 1$, and thus $u$ is identically equal to zero.

By a standard minimizing argument with respect to the parameter $s\ge 1$, we obtain the single logarithmic stability. Let
\[
\Phi_{\eta,c} (r)=r^{-1}\chi_{]0,e^c]}(r)+(\log r)^{-(2-\eta)/2}\chi_{]e^c,\infty [}(r),\quad 0\le\eta<2,\; c>0.
\]
Here $\chi_J$ denotes the characteristic function of the interval $J$.

\begin{cor}
Let $0\le\eta<2$ and $\zeta:=(g,X,p,B,\Omega,\eta)$. Then there exist constants $\mathbf{c}=\mathbf{c}(\zeta)>0$ and $c=c(\zeta)>0$  such that for any $u\in H^2(D)\setminus \{0\}$  we have
\[
\mathbf{c}\|u\|_{H^\eta(D)}\le \Phi_{\eta,c} \left(\|u\|_{H^2(D)}/\mathcal{C}(u)\right)\|u\|_{H^2(D)}.
\]
\end{cor}

%\begin{proof}
%When $\|u\|_{H^2(D)}>e^c\mathcal{C}(u)>0$, we choose $s\ge 1$ so that
%\[
%(e^c\mathcal{C}(u)\le )e^{cs}\mathcal{C}(u)=s^{-(2-\eta)/2}\|u\|_{H^2(D)}(\le \|u\|_{H^2(D)}),
%\]
%that is,
%\[
%s^{(2-\eta)/2}e^{cs}=\frac{\|u\|_{H^2(D)}}{\mathcal{C}(u)}(>e^c>1).
%\]
%Hence, there exists $c'>c$ independent of $s\ge 1$ such that
%\[
%c's\ge \log\frac{\|u\|_{H^2(D)}}{\mathcal{C}(u)}.
%\]
%Therefore, we have
%\[
%\mathbf{c}\|u\|_{H^\eta(D)}\le 2s^{-(2-\eta)/2}\|u\|_{H^2(D)}\le 2\|u\|_{H^2(D)}c'^{(2-\eta)/2}\left(\log\frac{\|u\|_{H^2(D)}}{\mathcal{C}(u)}\right)^{-(2-\eta)/2}.
%\]
%When $\|u\|_{H^2(D)}\le e^c\mathcal{C}(u)$, we have
%\[
%\mathbf{c}\|u\|_{H^\eta(D)}\le \left(e^{cs}+e^cs^{-(2-\eta)/2}\right)\mathcal{C}(u)=2e^c\mathcal{C}(u)
%\]
%by taking $s=1$.
%\end{proof}

\begin{remark}
The same results hold for the interior case as well. More precisely, let $B$ be a $C^2$ bounded domain and $\Omega\Subset B$ be a $C^{0,1}$ bounded open set such that $D:=B\setminus\overline{\Omega}$ is connected. When the measurement is made on $\Gamma:=\partial \Omega$, we can prove the same logarithmic stability as the one in Theorem \ref{interpolation_inequality}. The interior problem of this type is discussed in Choulli-Takase \cite{Choulli2025}.
\end{remark}

\section{Proof of Theorem \ref{interpolation_inequality}}

Let $B$, $\Omega$ and $D$ be  as in Section \ref{section1}. Recall that $\Gamma=\partial\Omega$ and $\mathcal{S}=\partial B$. Then, according to \cite[Lemma 3.1]{Choulli2025}, there exists $\phi\in C^2(\overline{D})$ satisfying
\begin{equation}\label{weight_function}
\begin{cases}
\phi>0,\quad \mbox{in}\; D,
\\
\phi_{|\mathcal{S}}=0,
\\
\displaystyle \delta:=\min_{\overline{D}}|\nabla\phi|>0.
\end{cases}
\end{equation}
Figure \ref{geometric_setup} below provides an illustration of this geometric configuration.

\begin{figure}[htbp]
\centering
\begin{tikzpicture}
% Draw the ellipse
\draw[fill=gray, fill opacity=0.5, line width=1pt] (-3.8,-0.8) -- (-3.5, 1.2) -- (-2.0, 2.2) -- ( 0.0, 1.7) -- ( 1.6, 2.4) -- ( 3.8, 2.0) -- ( 3.2, 0.8) -- ( 4.8,-0.2) -- ( 3.8,-1.8) -- ( 2.0,-2.5) -- ( 0.3,-2.0) -- (-1.3,-2.6) -- (-3.7,-1.9) -- cycle;
\draw[fill=white ] (0,0) ellipse (2cm and 1cm); 

% Denote the domain and boundary
\node at (0.3,0.5) {$B$};
\node at (1.7,0.2) {$\mathcal{S}$};
\node at (2,1.5) {$D:=\Omega\setminus \overline{B}$};    
\node at (4.2,0.5) {$\Gamma$};
%\node at (-1.54,0.8) {$\Omega$};

% Draw the axis
\draw[->] (-4.5,0) -- (5,0); % x-axis
\draw[->] (0,-2.5) -- (0,2.5); % y-axis
\foreach \x in {-2,0,2}
\draw (\x,-0.1) -- (\x,0.1); % Scale on the x-axis
\foreach \y in {-1,1}
\draw (-0.1,\y) -- (0.1,\y); % Scale on the y-axis
%\node[below] at (4.5,0) {$x$}; 
%\node[left] at (0,2) {$y$};
\end{tikzpicture}
\caption{Illustration of the geometric setting: $\phi=0$ on $\mathcal{S}$, while the Cauchy data are given on $\Gamma$.}
\label{geometric_setup}
\end{figure}

We will use  the following Carleman type inequality to prove Theorem \ref{interpolation_inequality}. 

\begin{prop}\label{global_Carleman_estimate}
Let $\zeta_0:=(g,X,p,D,\phi)$, $\varphi:=e^{\gamma\phi}$ and $\sigma:=s\gamma\varphi$. There exist constants $\gamma_\ast=\gamma_\ast(\zeta_0)>0$, $s_\ast=s_\ast (\zeta_0)>0$ and $\mathbf{c}=\mathbf{c}(\zeta_0)>0$ such that for any $\gamma \ge \gamma_\ast$, $s\ge s_\ast$ and $u\in H^2(D)$ we have
\begin{align*}
&\mathbf{c}\left(\int_D e^{2s\varphi}\sigma(\gamma|\nabla_g u|_g^2+\gamma\sigma^2|u|^2)dV_g+\int_\mathcal{S} e^{2s\varphi}\sigma(|\partial_{\nu_g}u|^2+\sigma^2|u|^2)dS_g\right)
\\
&\hskip 3cm \le \int_D e^{2s\varphi}|Pu|^2dV_g+\int_\Gamma e^{2s\varphi}\sigma(|\nabla_g u|_g^2+\sigma^2|u|^2)dS_g
\\
&\hskip 8.5cm+\int_\mathcal{S} e^{2s\varphi}\sigma|\nabla_{\tau_g} u|_g^2dS_g.
\end{align*}
\end{prop}

Note that only the integral term for the tangential derivative of $u$ on $\mathcal{S}$ appears in the right-hand side. Proposition \ref{global_Carleman_estimate} is proved by \cite[Proposition 2.1]{Choulli2025} and the proof is shown in Appendix \ref{appendix} for the sake of completeness.

\begin{proof}[Proof of Theorem \ref{interpolation_inequality}]

Fix $\gamma\ge \gamma_\ast$ in Proposition \ref{global_Carleman_estimate} and pick $u\in H^2(D)$. Henceforth, $\mathbf{c}_0=\mathbf{c}_0(\zeta_0)>0$ and $c_0=c_0(\zeta_0)>0$ denote generic constants, where $\zeta_0:=(g,X,p,B,\Omega)$. By \eqref{weight_function} and $\varphi_{|\mathcal{S}}=1$, the third term of the right-hand side in Proposition \ref{global_Carleman_estimate} can be written by
\begin{align*}
&\int_\mathcal{S} e^{2s\varphi}\sigma|\nabla_{\tau_g} u|^2 dS_g=e^{2s}s\gamma\|\nabla_{\tau_g}u\|_{L^2(\mathcal{S})}^2
\\
&\hskip 4cm\le e^{2s}s\gamma \|u\|_{H^1(\mathcal{S})}^2\le \mathbf{c}_0e^{2s}s\gamma \|u\|_{H^{2}(D)}^2.
\end{align*}

Let $\mathcal{C}=\mathcal{C}(u)$. Applying Proposition \ref{global_Carleman_estimate}, we get
\begin{align*}
&\mathbf{c}_0\left(\int_D e^{2s\varphi}s(|\nabla_g u|^2+s^2|u|^2)dV_g+\int_\mathcal{S} e^{2s\varphi}s(|\partial_{\nu_g}u|^2+s^2|u|^2)dS_g\right)
\\
&\hskip 6.5cm \le e^{c_0s}\mathcal{C}^2+se^{2s}\|u\|_{H^{2}(D)}^2,\quad s\ge s_\ast,
\end{align*}
where $s_\ast=s_\ast(\zeta_0)>0$ is a constant.

By
\begin{align*}
&\mathbf{c}_0\left(\int_D e^{2s\varphi}s(|\nabla_g u|^2+s^2|u|^2)dV_g+\int_\mathcal{S} e^{2s\varphi}s(|\partial_{\nu_g}u|^2+s^2|u|^2)dS_g\right)
\\
&\hskip 8.5cm \ge s^3e^{2s}\int_D |u|^2dV_g,
\end{align*}
we have
\begin{align*}
\mathbf{c}_0\|u\|_{L^2(D)}^2&\le e^{c_0s}\mathcal{C}^2+s^{-2}\|u\|_{H^{2}(D)}^2
\\
&\le \left(e^{(c_0/2)s}\mathcal{C}+s^{-1}\|u\|_{H^{2}(D)}\right)^2,\quad s\ge s_\ast,
\end{align*}
which implies
\begin{equation}\label{L^2}
\mathbf{c}_0\|u\|_{L^2(D)}\le e^{c_0s}\mathcal{C}+s^{-1}\|u\|_{H^{2}(D)},\quad s\ge s_\ast.
\end{equation}
By replacing $s$ by $s/s_\ast$,  we see that the above inequality holds for all $s\ge 1$, which means that the expected inequality holds when $\eta=0$.

When $0<\eta<2$, by the definition of $H^\eta(\Omega)$ as an interpolation space (e.g., \cite[p. 40, (9.1)]{Lions1972I}), the interpolation inequality (e.g., \cite[p. 19, Proposition 2.3]{Lions1972I}) and Young's inequality, for any $\varepsilon>0$ we have
\[
c_1\|u\|_{H^\eta(D)}\le \varepsilon^{-1}\|u\|_{L^2(D)}+\varepsilon^{(2-\eta)/\eta}\|u\|_{H^2(D)},
\]
where $c_1=c_1(D,\eta)>0$ is a constant. We take $\varepsilon:=s^{-\eta/2}$ and apply \eqref{L^2} to obtain the expected inequality
\[
\mathbf{c}\|u\|_{H^\eta(D)}\le e^{cs}\mathcal{C}+s^{-(2-\eta)/2}\|u\|_{H^2(D)},\quad s\ge 1,
\]
where $\mathbf{c}=\mathbf{c}(\zeta)>0$ and $c=c(\zeta)>0$ are constants.
\end{proof}

\section{Some extensions}

\subsection{$C^{j}$ bounded domain for $j\ge 2$}

We mention an extension for the case of an arbitrary $C^j$ bounded domain for $j\ge 2$. Although the following lemma is proved by the similar argument as in \cite[Lemma 1.9]{Choulli}, we provide the proof for the sake of completeness.

\begin{lem}\label{Cj}
Let $j\ge 2$ be an integer, $D$ be a $C^j$ bounded domain of $\mathbb{R}^n$ and $\Gamma$ be a nonempty open subset of $\partial D$. Then there exists $\phi\in C^j(\overline{D})$ satisfying \eqref{weight_function} with $\mathcal{S}:=\partial D\setminus\Gamma$.
\end{lem}
\begin{proof}
Let $\widetilde{D}$ be a $C^j$ bounded domain of $\mathbb{R}^n$ satisfying $D\subset \widetilde{D}$, $\mathrm{Int}(\widetilde{D}\setminus D)\neq \emptyset$ and $\partial D \setminus \Gamma \subset \partial\widetilde{D}$. Pick $\omega\subset \mathrm{Int}(\widetilde{D}\setminus D)$ be a nonempty open subset. According to \cite[Theorem 9.4.3]{Tucsnak2009}, there exists $\psi\in C^j(\overline{\widetilde{D}})$ satisfying
\[
\psi>0\; \text{in}\; \widetilde{D},\quad \psi=0\; \text{on}\; \partial\widetilde{D}\quad \text{and}\quad |\nabla\psi|>0\; \text{in}\; \overline{\widetilde{D}\setminus\omega}.
\]
We find that $\phi:=\psi_{\overline{D}}\in C^j(\overline{D})$ satisfies \eqref{weight_function} with $\mathcal{S}:=\partial D \setminus \Gamma$.
\end{proof}

Although results similar to the following theorem have been obtained in Bourgeois \cite{Bourgeois2010} (Laplace operator) and Choulli \cite{Choulli2024a} (general case), the proof can be simplified as shown in Theorem \ref{interpolation_inequality}.

\begin{thm}\label{ii_Cj}
Let $j\ge 2$ be an integer, $D$ be a $C^j$ bounded domain of $\mathbb{R}^n$, $\Gamma$ be a nonempty open subset of $\partial D$, $0\le\eta<2$ and $\zeta:=(g,X,p,D,\Gamma,\eta)$. Then there exist constants $\mathbf{c}=\mathbf{c}(\zeta)>0$ and $c=c(\zeta) >0$ such that for any $u\in H^2(D)$ and $s\ge 1$ we have
\[
\mathbf{c}\|u\|_{H^\eta(D)}\le e^{cs}\mathcal{C}(u)+s^{-(2-\eta)/2}\|u\|_{H^2(D)}.
\]
\end{thm}

\begin{proof}
Thanks to to Lemma \ref{Cj}, Proposition \ref{global_Carleman_estimate} remains valid also in the case $\mathcal{S}:=\partial D\setminus \Gamma$. The proof is completed in a manner analogous to that of Theorem \ref{interpolation_inequality}. 
\end{proof}

\subsection{Extension to the Stokes equation}

Let $n\ge 2$, $B$ and $\Omega$ be the same  as in Section \ref{section1}. Let $a=(a^1,\ldots,a^n)\in W^{1,\infty}(D;\mathbb{C}^n)$ and consider the Stokes equation
\begin{equation}\label{Stokes}
\begin{cases}
-\Delta u+(a\cdot\nabla)u+\nabla p=0\quad &\text{in}\; D,
\\
\mathrm{div}(u)=0\quad &\text{in}\; D,
\end{cases}
\end{equation}
where $u=(u^1,\ldots,u^n)$. 

We prove an optimal quantitative unique continuation from the Cauchy data  $(u_{|\Gamma},p_{|\Gamma},\partial_\nu u_{|\Gamma},\partial_\nu p_{|\Gamma})$, where we recall that $\Gamma=\partial\Omega$. Note that the pressure is determined only up to an additive constant in the Stokes equation in general. Since the Cauchy data include the trace $p_{|\Gamma}$, this additive constant is uniquely fixed. A similar result for a smooth domain has been established by Boulakia-Egloffe-Grandmont \cite{Boulakia2013}.

Define
\[
\mathfrak{C}(u,p)=\|u\|_{H^1(\Gamma)^n}+\|p\|_{H^1(\Gamma)}+\|\partial_\nu u\|_{L^2(\Gamma)^n}+\|\partial_\nu p\|_{L^2(\Gamma)}.
\]

\begin{thm}\label{interpolation_Stokes}
Let $0\le\eta<2$ and $\zeta:=(a,B,\Omega,\eta)$. Then, there exists $\mathbf{c}=\mathbf{c}(\zeta)>0$ and $c=c(\zeta)>0$ such that for any $(u,p)\in H^2(D)^{n+1}$ satisfying \eqref{Stokes} we have
\[
\mathbf{c}\|(u,p)\|_{H^\eta(D)^{n+1}}\le e^{cs}\mathfrak{C}(u,p)+s^{-(2-\eta)/2}\|(u,p)\|_{H^2(D)^{n+1}},\quad s\ge 1.
\]
\end{thm}

\begin{remark}
For a general $C^j$ domain $D$ with an integer $j\ge 2$, a result similar to that described in Theorem \ref{ii_Cj} is valid for the Stokes equation \eqref{Stokes}.
\end{remark}

As for Theorem \ref{interpolation_inequality}, Theorem \ref{interpolation_Stokes} yields the following single logarithmic stability inequality, where $\Phi_{\eta,c}$, $0\le \eta <2$ and $c>0$, is defined in Section \ref{section1}.

\begin{cor}
Let $0\le\eta<2$ and $\zeta:=(a,B,\Omega,\eta)$. Then, there exist constants $\mathbf{c}=\mathbf{c}(\zeta)>0$ and $c=c(\zeta)>0$ such that for any $(u,p)\in H^2(D)^{n+1}\setminus\{0\}$ we have
\[
\mathbf{c}\|(u,p)\|_{H^\eta(D)^{n+1}}\le \Phi_{\eta,c}\left(\|(u,p)\|_{H^2(D)^{n+1}}/\mathfrak{C}(u,p)\right)\|(u,p)\|_{H^2(D)^{n+1}}.
\]
\end{cor}

\begin{proof}[Proof of Theorem \ref{interpolation_Stokes}]
In this proof, $\mathbf{c}_0=\mathbf{c}_0(a,B,\Omega)>0$ and $c_0=c_0(a,B,\Omega)>0$ will denote generic constants. Set $v:=(u,p)\in H^2(D)^{n+1}$, where $u$ and $p$ satisfy \eqref{Stokes}. Since $p$ satisfies in the distributional sense
\[
\Delta p=\mathrm{div}(\Delta u-(a\cdot \nabla)u)=-\partial_k a^j \partial_j u^k\quad \text{in}\; D,
\]
where we used the equation $\mathrm{div}(u)=0$, $v$ satisfies
\[
\Delta v=((a\cdot\nabla)u+\nabla p,-\partial_k a^j\partial_j u^k)\quad \text{in}\; D.
\]
From this equation, we have
\begin{equation}\label{inequality_v}
\mathbf{c}_0|\Delta v|^2\le |\nabla u|^2+|\nabla p|^2=|\nabla v|^2\quad \text{in}\; D,
\end{equation}
where the generic constant $\mathbf{c}_0$ depends on $\|a\|_{W^{1,\infty}(D)^n}$.

Fix $\gamma\ge \gamma_\ast$ in Proposition \ref{global_Carleman_estimate}. Let $\mathfrak{C}=\mathfrak{C}(v)$. Applying Proposition \ref{global_Carleman_estimate} to each component of $v=(v^1,\ldots,v^{n+1})$ with $g=\mathbf{I}$ and $P=-\Delta$, we get
\begin{align*}
&\mathbf{c}_0\left(\int_D e^{2s\varphi}s(|\nabla v^k|^2+s^2|v^k|^2)dx+\int_\mathcal{S} e^{2s\varphi}s(|\partial_{\nu}v^k|^2+s^2|v^k|^2)dS\right)
\\
&\hskip .2cm \le \int_D e^{2s\varphi}|\Delta v^k|^2dx+e^{c_0s}\left(\|v^k\|_{H^1(\Gamma)}^2+\|\partial_\nu v^k\|_{L^2(\Gamma)}^2\right)+se^{2s}\|v^k\|_{H^2(D)}^2,\quad s\ge s_\ast,
\end{align*}
where $s_\ast=s_\ast(a,B,\Omega)>0$ is a constant and $k\in\{1,\ldots,n+1\}$. By summing the above inequalities side by side on $k\in \{1,\ldots,n+1\}$, we obtain
\begin{align*}
&\mathbf{c}_0\left(\int_D e^{2s\varphi}s(|\nabla v|^2+s^2|v|^2)dx+\int_\mathcal{S} e^{2s\varphi}s(|\partial_{\nu}v|^2+s^2|v|^2)dS\right)
\\
&\hskip .2cm \le \int_D e^{2s\varphi}|\Delta v|^2dx+e^{c_0s}\left(\|v\|_{H^1(\Gamma)}^2+\|\partial_\nu v\|_{L^2(\Gamma)}^2\right)+se^{2s}\|v\|_{H^2(D)}^2,\quad s\ge s_\ast,
\end{align*}
Since we have by \eqref{inequality_v}
\[
\mathbf{c}_0 \int_D e^{2s\varphi}|\Delta v|^2dx\le \int_D e^{2s\varphi}|\nabla v|^2dx,
\]
it follows that
\begin{align*}
&\mathbf{c}_0\left(\int_D e^{2s\varphi}s(|\nabla v|^2+s^2|v|^2)dx+\int_\mathcal{S} e^{2s\varphi}s(|\partial_{\nu}v|^2+s^2|v|^2)dS\right)
\\
&\hskip 3cm \le \int_D e^{2s\varphi}|\nabla v|^2dx+c^{c_0s}\mathfrak{C}^2+se^{2s}\|v\|_{H^2(D)}^2,\quad s\ge s_\ast.
\end{align*}

By modifying $s_\ast$, we may and do assume $\mathbf{c}_0s-1\ge \mathbf{c}_0s/2$ so that
\[
\int_D e^{2s\varphi}|\nabla v|^2dx\le \frac{\mathbf{c}_0}{2}\int_D e^{2s\varphi}s|\nabla v|^2dx,\quad s\ge s_\ast.
\]

Therefore, we have
\begin{align*}
&\mathbf{c}_0\left(\int_D e^{2s\varphi}s(|\nabla v|^2+s^2|v|^2)dx+\int_\mathcal{S} e^{2s\varphi}s(|\partial_{\nu}v|^2+s^2|v|^2)dS\right)
\\
&\hskip 6cm \le e^{c_0s}\mathfrak{C}^2+se^{2s}\|v\|_{H^2(D)}^2,\quad s\ge s_\ast.
\end{align*}

By using
\begin{align*}
&\mathbf{c}_0\left(\int_D e^{2s\varphi}s(|\nabla v|^2+s^2|v|^2)dx+\int_\mathcal{S} e^{2s\varphi}s(|\partial_{\nu}v|^2+s^2|v|^2)dS\right)
\\
&\hskip 8cm \ge s^3e^{2s}\int_D |v|^2dx,
\end{align*}
we obtain
\[
\mathbf{c}_0\|v\|_{L^2(D)}\le e^{c_0s}\mathfrak{C}+s^{-1}\|v\|_{H^2(D)},\quad s\ge s_\ast.
\]
By replacing $s$ by $s/s_\ast$,  we see that the above inequality holds for all $s\ge 1$, which means that the expected inequality holds when $\eta=0$.

When $0<\eta<2$, by the interpolation inequality and Young's inequality as in the proof of Theorem \ref{interpolation_inequality}, for any $\varepsilon>0$ we have
\[
\mathbf{c}\|v\|_{H^\eta(D)}\le \varepsilon^{-1}\|v\|_{L^2(D)}+\varepsilon^{(2-\eta)/\eta}\|v\|_{H^2(D)},
\]
where $\mathbf{c}=\mathbf{c}(\zeta)>0$ is a constant. We complete the proof by taking $\varepsilon:=s^{-\eta/2}$ and applying the above inequality.
\end{proof}

\appendix

\section{Proof of Proposition \ref{global_Carleman_estimate}}\label{appendix}
\begin{proof}
According to \cite[Proposition 2.1]{Choulli2025}, there exist constants $\gamma_\ast=\gamma_\ast(g,D,\phi)\ge 1$, $s_\ast=s_\ast (g,D,\phi)\ge 1$ and $\mathbf{c}_0=\mathbf{c}_0(g,D,\phi)>0$ such that for any $\gamma \ge \gamma_\ast$, $s\ge s_\ast$ and $v\in H^2(D;\mathbb{R})$, we have
\begin{align*}
&\mathbf{c}_0\left(\int_D e^{2s\varphi}\sigma(\gamma|\nabla_g v|_g^2+\gamma\sigma^2 v^2)dV_g+\int_\mathcal{S} e^{2s\varphi}\sigma((\partial_{\nu_g}v)^2+\sigma^2 v^2)dS_g\right)
\\
&\hskip 3cm \le \int_D e^{2s\varphi}(\Delta_g v)^2dV_g+\int_\Gamma e^{2s\varphi}\sigma(|\nabla_g v|_g^2+\sigma^2 v^2)dS_g
\\
&\hskip 8.5cm+\int_\mathcal{S} e^{2s\varphi}\sigma|\nabla_{\tau_g} v|_g^2dS_g.
\end{align*}
By adding the inequalities applied to $v := \Re u$ and to $v := \Im u$ together, we obtain for any $\gamma \ge \gamma_\ast$, $s\ge s_\ast$ and $u\in H^2(D)$,
\begin{align*}
&\mathbf{c}_0\left(\int_D e^{2s\varphi}\sigma(\gamma|\nabla_g u|_g^2+\gamma\sigma^2|u|^2)dV_g+\int_\mathcal{S} e^{2s\varphi}\sigma(|\partial_{\nu_g} u|^2+\sigma^2|u|^2)dS_g\right)
\\
&\hskip 3cm \le \int_D e^{2s\varphi}|\Delta_g u|^2dV_g+\int_\Gamma e^{2s\varphi}\sigma(|\nabla_g u|_g^2+\sigma^2|u|^2)dS_g
\\
&\hskip 8.5cm+\int_\mathcal{S} e^{2s\varphi}\sigma|\nabla_{\tau_g} u|_g^2dS_g.
\end{align*}
Since $-\Delta_g u=Pu-\langle X, \nabla_g u\rangle-pu$, we have
\begin{align}\label{Xp}
& \int_D e^{2s\varphi}|\Delta_g u|^2dV_g\le 3\int_D e^{2s\varphi}|P u|^2dV_g
\\
&\hskip 3cm +3\int_D e^{2s\varphi}|\langle X,\nabla_g u\rangle|^2dV_g+3\int_D e^{2s\varphi}|pu|^2dV_g. \nonumber
\end{align}
Moreover, we get
\begin{equation}\label{X}
\int_D e^{2s\varphi}|\langle X,\nabla_g u\rangle|^2dV_g\le \|X\|_{L^\infty(D)^n}^2\int_D e^{2s\varphi}|\nabla_g u|^2dV_g
\end{equation}
and by H\"older's inequality
\[
\int_De^{2s\varphi}|pu|^2dV_g \le \|p\|_{L^r(D)}^2\|e^{2s\varphi}u\|_{L^\frac{2r}{r-2}(D)}^2\le \mathbf{c}_\ast^2 \|p\|_{L^r(D)}^2\|e^{s\varphi}u\|_{H^1(D)}^2,
\]
where $\mathbf{c}_\ast=\mathbf{c}_\ast (D)>0$ denotes the constant of the embedding $H^1(D)\hookrightarrow L^\frac{2r}{r-2}(D)$. Therefore, by modifying the constant $\mathbf{c}_0=\mathbf{c}_0(g,D,\phi)>0$, we get
\begin{equation}\label{p}
\mathbf{c_0} \int_De^{2s\varphi} |pu|^2dV_g  \le \|p\|_{L^r(D)}^2\int_D e^{2s\varphi}(|\nabla_g u|^2+\sigma^2|u|^2)dV_g.
\end{equation}
Applying \eqref{X} and \eqref{p} to \eqref{Xp}, we have
\[
\int_D e^{2s\varphi}|\Delta_g u|^2dV_g\le \mathbf{c}\left(\int_D e^{2s\varphi}|P u|^2dV_g+\int_D e^{2s\varphi}(|\nabla_g u|^2+\sigma^2 |u|^2)dV_g\right),
\]
where $\mathbf{c}=\mathbf{c}(\zeta_0)>0$ denotes a generic constant. By modifying $\gamma_\ast=\gamma_\ast(\zeta_0)\ge 1$, we may and do assume that
\[
\mathbf{c}_0 \gamma-\mathbf{c}\ge \mathbf{c}_0 \gamma /2,\quad \gamma\ge \gamma_\ast.
\]
In this case, by $\sigma\ge 1$, we have for any $\gamma\ge \gamma_\ast$ and $s\ge s_\ast$,
\begin{align*}
&\mathbf{c}_0\left(\int_D e^{2s\varphi}\sigma(\gamma|\nabla_g u|_g^2+\gamma\sigma^2|u|^2)dV_g+\int_\mathcal{S} e^{2s\varphi}\sigma(|\partial_{\nu_g} u|^2+\sigma^2|u|^2)dS_g\right)
\\
&\hskip 1cm \le \mathbf{c}\int_D e^{2s\varphi}|P u|^2dV_g+(\mathbf{c}_0/2)\int_D e^{2s\varphi}\sigma(\gamma|\nabla_g u|_g^2+\gamma\sigma^2|u|^2)dV_g
\\
&\hskip 3cm +\int_\Gamma e^{2s\varphi}\sigma(|\nabla_g u|_g^2+\sigma^2|u|^2)dS_g
+\int_\mathcal{S} e^{2s\varphi}\sigma|\nabla_{\tau_g} u|_g^2dS_g,
\end{align*}
which completes the proof by modifying the generic constant $\mathbf{c}=\mathbf{c}(\zeta_0)>0$.
\end{proof}

\section*{Acknowledgement}
This work was supported by JSPS KAKENHI Grant Numbers JP25K17280 and JP23KK0049.

\end{document}